\documentclass[11pt]{article}

\usepackage{amsmath,amsfonts,amssymb,amsthm,graphics,amscd}
\usepackage{latexsym,mathrsfs}

\newtheorem{theorem}{Theorem}[section]
\newtheorem{lemma}[theorem]{Lemma}
\newtheorem{corollary}[theorem]{Corollary}

\newtheorem{example}[theorem]{Example}

\newtheorem{remark}[theorem]{Remark}

\def\DD{\mathbb{D}}

\def\NN{\mathbb{N}}
\def\CC{\mathbb{C}}
\def\XX{\mathcal{X}}
\def\YY{\mathcal{Y}}
\def\HH{\mathcal{H}}
\def\KK{\mathcal{K}}

\numberwithin{equation}{section}

\def\int{{\rm int}}
\def\asc{{\rm asc}}
\def\dsc{{\rm dsc}}
\def\acc{{\rm acc}}

\def\dim{{\rm dim}}

\def\codim{{\rm codim}}

\title{Generalized Drazin invertibility of operator matrices}

\author{Milo\v s D. Cvetkovi\'c}

\date{}

\begin{document}

\maketitle

\begin{abstract}
\noindent We study the generalized Drazin invertibility as well as
the Drazin and ordinary invertibility of an operator matrix
$M_C=\left(
\begin{array}{cc} A & C \\
0 & B\end{array} \right)$ acting on a Banach space $\XX \oplus \YY$
or on a Hilbert space $\HH \oplus \KK$. As a consequence some recent
results are extended.
\end{abstract}



\section{Introduction and Preliminaries}

Let $\XX$ and $\YY$ be infinite dimensional Banach spaces. The set
of all bounded linear operators from $\XX$ to $\YY$ will be denoted
by $L(\XX,\YY)$. For simplicity, we write $L(\XX)$ for $L(\XX,
\XX)$. The set
\[\XX \oplus \YY=\{(x,y): x \in \XX, y \in \YY \} \]
is a vector space with standard addition and multiplication by
scalars. Under the norm
\[||(x,y)||=(||x||^2+||y||^2)^{\frac{1}{2}}\]
$\XX \oplus \YY$ becomes a Banach space. If $\XX_1$ and $\YY_1$ are
closed subspaces of $\XX$ and $\YY$ respectively, then we will use
sometimes notation $\left( \begin{array}{c} \XX_1 \\
\YY_1 \end{array} \right)$ instead of $\XX_1 \oplus \YY_1$. If $A
\in L(\XX)$, $B \in L(\YY)$ and $C \in L(\YY, \XX)$ are given, then
\[\left(
\begin{array}{cc} A & C \\
0 & B\end{array} \right): \XX \oplus \YY \to \XX \oplus \YY \]
represents a bounded linear operator on $\XX \oplus \YY$ and it is
called {\em upper triangular operator matrix}. If $A$ and
$B$ are given then we write $M_C=\left( \begin{array}{cc} A & C \\
0 & B \end{array} \right)$ in order to emphasize dependence on $C$.

Let $\mathbb{N} \, (\mathbb{N}_0)$ denote the set of all positive
(non-negative) integers, and let $\mathbb{C}$ denote the set of all
complex numbers. Given $T \in L(\XX, \YY)$, we denote by $N(T)$ and
$R(T)$ the {\em kernel} and the {\em range} of $T$. The numbers
$\alpha(T)=\dim N(T)$ and $\beta(T)=\codim R(T)$ are {\em nullity}
and {\em deficiency} of $T \in L(\XX, \YY)$ respectively. The set
\[\sigma(T)=\{\lambda \in \CC: T-\lambda \; \text{is not invertible} \} \]
is the {\em spectrum} of $T \in L(\XX)$. An operator $T \in L(\XX)$
is {\em bounded below} if there exists some $c>0$ such that
$c\|x\|\leq \|Tx\| \; \; \text{for every} \; \; x \in \XX$. It is
worth mentioning that the set of bounded below operators and the set
of left invertible operators coincide in the Hilbert space setting.
Also, the set of surjective operators and the set of right
invertible operators coincide on a Hilbert space. Recall that $T \in
L(\XX)$ is nilpotent when $T^n=0$ for some $n \in \NN$, while $T \in
L(\XX)$ is {\em quasinilpotent} if $T-\lambda$ is invertible for all
complex $\lambda\ne 0$. The {\em ascent} of $T \in L(\XX)$ is
defined as $\asc(T)=\inf\{n \in \mathbb{N}_0:N(T^n)=N(T^{n+1})\}$,
and {\em descent} of $T$ is defined as $\dsc(T)=\inf\{n \in
\mathbb{N}_0:R(T^n)=R(T^{n+1})\}$, where the infimum over the empty
set is taken to be infinity. If $K \subset \mathbb{C}$, $\acc \, K$
is the set of accumulation points of $K$.

If $M$ is a subspace of $\XX$ such that $T(M) \subset M$, $T \in
L(\XX)$, it is said that $M$ is {\em $T$-invariant}. We define
$T_M:M \to M$ as $T_Mx=Tx, \, x \in M$.  If $M$ and $N$ are two
closed $T$-invariant subspaces of $\XX$ such that $\XX=M \oplus N$,
we say that $T$ is {\em completely reduced} by the pair $(M,N)$ and
it is denoted by $(M,N) \in Red(T)$. In this case we write $T=T_M
\oplus T_N$ and say that $T$ is a {\em direct sum} of $T_M$ and
$T_N$.

An operator $T \in L(\XX)$ is said to be {\em Drazin invertible}, if
there exists $S \in L(\XX)$ and some $k \in \NN$ such that
\[TS=ST, \; \; \; STS=S, \; \; \; T^kST=T^k.\]
It is a classical result that necessary and sufficient for $T \in
L(\XX)$ to be Drazin invertible is that $T=T_1 \oplus T_2$ where
$T_1$ is invertible and $T_2$ is nilpotent; see \cite{K, lay}. The
{\em Drazin spectrum} of $T$ is defined as
\[\sigma_D(T)=\{\lambda \in \CC: T-\lambda \; \text{is not Drazin invertible} \}
\] and it is compact \cite[Proposition 2.5]{Ber1}. An operator $T
\in L(\XX)$ is {\em left Drazin invertible} if $\asc(T)<\infty$ and
if $R(T^{\asc(T)+1})$ is closed and $T \in L(\XX)$ is {\em right
Drazin invertible} if $\dsc(T)<\infty$ and if $R(T^{\dsc(T)})$ is
closed.

J. Koliha extended the concept of Drazin invertibility
\cite{koliha}. An operator $T \in L(\XX)$ is said to be {\em
generalized Drazin invertible}, if there exists $S \in L(\XX)$ such
that
\[TS=ST, \; \; \; STS=S, \; \; \; TST-T \; \; \text{is quasinilpotent}.\]

\noindent According to \cite[Theorem 4.2 and Theorem 7.1]{koliha},
$T \in L(\XX)$ is generalized Drazin invertible if and only if $0
\not \in \acc \, \sigma(T)$, and it is exactly when $T=T_1 \oplus
T_2$ with $T_1$ invertible and $T_2$ quasinilpotent. Naturally, the
set
\[\sigma_{gD}(T)=\{\lambda \in \CC: T-\lambda \; \text{is not generalized Drazin  invertible} \} \]
is the {\em generalized Drazin spectrum} of $T$.

We also recall definitions of the following spectra of $T \in
L(\XX)$:\par

\medskip

\noindent the point spectrum: $\sigma_p(T)=\{\lambda \in \CC:
T-\lambda \; \text{is not injective} \}$,

\smallskip

\noindent the defect spectrum: $\sigma_d(T)=\{\lambda \in \CC:
\overline{R(T-\lambda)} \neq \XX \}$,

\smallskip

\noindent the approximate point spectrum: $\sigma_{ap}(T)=\{\lambda
\in \CC: T-\lambda \; \text{is not bounded below} \}$,\par

\smallskip

\noindent the surjective spectrum: $\sigma_{su}(T)=\{\lambda \in
\CC: T-\lambda \; \text{is not surjective} \}$,\par

\smallskip

\noindent the left spectrum: $\sigma_l(T)=\{\lambda \in \CC:
T-\lambda \; \text{is not left invertible} \}$,\par

\smallskip

\noindent the right spectrum: $\sigma_r(T)=\{\lambda \in \CC:
T-\lambda \; \text{is not right invertible} \}$.\par

\medskip

Let $M$ and $L$ be two closed linear subspaces of $\XX$ and set
\[\delta(M,L)=\sup \{dist(u,L): u \in M, \, \|u\|=1\},\]
in the case that $M \neq \{0\}$, otherwise we define $\delta(\{0\},
L)=0$ for any subspace $L$. The {\em gap} between $M$ and $L$ is
defined by
\[\hat{\delta}(M,L)=\max\{\delta(M,L), \delta(L,M)\}.\]
It is known that \cite[corollary 10.10]{Mu}
\begin{equation}\label{gp1}
  \hat{\delta}(M,L)<1\Longrightarrow\dim M=\dim L.
\end{equation}

An operator $T \in L(\XX)$ is {\em semi-regular} if $R(T)$ is closed
and if one of the following equivalent statements holds:\par

\noindent {\rm (i)} $N(T) \subset R(T^m)$ for each $m \in \NN$;\par

\smallskip

\noindent {\rm (ii)} $N(T^n) \subset R(T)$ for each $n \in \NN$.\par

\smallskip

\noindent It is clear that left and right invertible operators are
semi-regular. V. M\"{u}ller shows (\cite[Corollary 12.4]{Mu}) that
if $T \in L(\XX)$, then $T$ is semi-regular and $0 \not \in \acc \,
\sigma(T)$ imply that $T$ is invertible. In particular,
\begin{align}
T \; \; \text{is left invertible and} \; \; 0 \not \in \acc \, \sigma(T) \; \; \Longrightarrow T \; \; \text{is invertible}, \label{leftsemi}\\
T \; \; \text{is right invertible and} \; \; 0 \not \in \acc \,
\sigma(T) \; \; \Longrightarrow T \; \; \text{is
invertible}.\label{rightsemi}
\end{align}

An operator $T \in L(\XX)$ is said to admit a {\em generalized Kato
decomposition}, abbreviated as GKD, if there exists a pair $(M,N)
\in Red(T)$ such that $T_M$ is semi-regular and $T_N$ is
quasinilpotent. A relevant case is obtained if we assume that $T_N$
is nilpotent. In this case $T$ is said to be of {\em Kato type}.

The invertibility, Drazin invertibility and generalized Drazin
invertibility of upper triangular operator matrices have been
studied by many authors, see for example \cite{Du, Han, ELA, Zhong,
maroko, Dragan, Dragan1, Zhang, pseudoBW}. In this article we study
primarily the generalized Drazin invertibility but also the Drazin
and ordinary invertibility of operator matrices by using the
technique that involves the gap theory; see the auxiliary results:
\eqref{gp1}, \eqref{leftsemi}, \eqref{rightsemi}, Lemma
˜\ref{lema1}, Lemma ˜\ref{lema3}. Let $A \in L(\HH)$ and $B \in
L(\KK)$, where $\HH$ and $\KK$ are separable Hilbert spaces, and
consider the following conditions:\par

\smallskip

\noindent {\rm (i)} $A=A_1 \oplus A_2$, $A_1$ is bounded below and
$A_2$ is quasinilpotent; \par \noindent{\rm (ii)} $B=B_1 \oplus
B_2$, $B_1$ is surjective and $B_2$ is quasinilpotent;
\par \noindent {\rm (iii)} There exists a constant $\delta>0$ such
that $\beta(A-\lambda)=\alpha(B-\lambda)$ for every $\lambda \in
\CC$ such that $0<|\lambda|<\delta$. \par

\smallskip

\noindent Our main result stands that if the conditions {\rm
(i)}-{\rm (iii)} are satisfied then there exists $C \in L(\KK, \HH)$
such that $M_C$ is generalized Drazin invertible. The converse is
also true under the assumption that $A$ and $B$ admit a GKD.
Moreover, we obtain the corresponding results concerning the Drazin
invertibility of operator matrices. Further, let $A \in L(\XX)$ and
$B \in L(\YY)$ where $\XX$ and $\YY$ are Banach spaces. It is also
shown that if $0$ is not an accumulation point of the defect
spectrum of $A$ or of the point spectrum of $B$ and if the operator
matrix $M_C=\left(
\begin{array}{cc} A & C \\
0 & B\end{array} \right)$ is invertible (resp. Drazin invertible,
generalized Drazin invertible) for some $C \in L(\YY, \XX)$ then $A$
and $B$ are both invertible (resp. Drazin invertible, generalized
Drazin invertible). What is more, we give several corollaries that
improve some recent results.

\section{Results}

From now on $\XX$ and $\YY$ will denote Banach spaces, while $\HH$
and $\KK$ will denote separable Hilbert spaces. We begin with
several auxiliary lemmas.

\begin{lemma} \label{lema1}
Let $T \in L(\XX)$ be semi-regular. Then there exists $\epsilon
>0$ such that $\alpha(T-\lambda)$ and $\beta(T-\lambda)$ are
constant for $|\lambda|<\epsilon$.
\end{lemma}
\begin{proof}
Let $T \in L(\XX)$ be semi-regular. The mapping $\lambda \to
N(T-\lambda)$ is continuous at the point $0$ in the gap metric; see
\cite[Theorem 1.38]{aiena} or \cite[Theorem 12.2]{Mu}. In
particular, there exists $\epsilon_1>0$ such that
$|\lambda|<\epsilon_1$ implies $\hat{\delta}(N(T), N(T-\lambda))<1$.
From \eqref{gp1} we obtain $\dim N(T)=\dim N(T-\lambda)$, i.e.
$\alpha(T)=\alpha(T-\lambda)$ for $|\lambda|<\epsilon_1$. Further,
$T^{\prime}$ is also semi-regular where $T^{\prime}$ is adjoint
operator of $T$ \cite[Theorem 1.19]{aiena}. As above we conclude
that $\alpha(T^{\prime})=\alpha(T^{\prime}-\lambda)$ on an open disc
centered at $0$. Since $T-\lambda$ has closed range for sufficiently
small $|\lambda|$ \cite[Theorem 1.31]{aiena}, it follows that there
exists $\epsilon_2>0$ such that
\[\beta(T)=\alpha(T^{\prime})=\alpha(T^{\prime}-\lambda)=\beta(T-\lambda)
\; \; \text{for} \; \; |\lambda|<\epsilon_2.\] We put
$\epsilon=\min\{\epsilon_1, \epsilon_2\}$, and the lemma follows.
\end{proof}

\begin{lemma} [\cite{Han}] \label{invertibility}
Let $A \in L(\XX)$ and $B \in L(\YY)$. If $A$ and $B$ are both
invertible, then $M_C$ is invertible for every $C \in L(\YY, \XX)$.
In addition, if $M_C$ is invertible for some $C \in L(\YY, \XX)$,
then $A$ is invertible if and only if $B$ is invertible.
\end{lemma}

\begin{lemma} [\cite{joint}] \label{lema3}

\noindent {\rm (a)} Let $T \in L(\XX)$. The following conditions are
equivalent:\par

\noindent {\rm (i)} There exists $(M,N) \in Red(T)$ such that $T_M$
is bounded below (resp. surjective) and $T_N$ is quasinilpotent;\par

\noindent {\rm (ii)} $T$ admits a GKD and $0 \not \in \acc \;
\sigma_{ap}(T)$ (resp. $0 \not \in \acc \; \sigma_{su}(T))$.

\medskip

\noindent {\rm (b)} Let $T \in L(\XX)$. The following conditions are
equivalent:\par

\noindent {\rm (i)} There exists $(M,N) \in Red(T)$ such that $T_M$
is bounded below (resp. surjective) and $T_N$ is nilpotent;\par

\noindent {\rm (ii)} $T$ is of Kato type and $0 \not \in \acc \;
\sigma_{ap}(T)$ (resp. $0 \not \in \acc \; \sigma_{su}(T))$.
\end{lemma}

We now prove our first main result.

\begin{theorem} \label{GD}
Let $A \in L(\HH)$ and $B \in L(\KK)$ be given operators on
separable Hilbert spaces $\HH$ and $\KK$, respectively, such
that:\par

\medskip

\noindent {\rm (i)} $A=A_1 \oplus A_2$, $A_1$ is bounded below and
$A_2$ is quasinilpotent; \par \noindent{\rm (ii)} $B=B_1 \oplus
B_2$, $B_1$ is surjective and $B_2$ is quasinilpotent; \par
\noindent {\rm (iii)} There exists a constant $\delta>0$ such that
$\beta(A-\lambda)=\alpha(B-\lambda)$ for every $\lambda \in \CC$
such that $0<|\lambda|<\delta$. \par

\smallskip

\noindent Then there exists an operator $C \in L(\KK, \HH)$ such
that $M_C$ is generalized Drazin invertible.
\end{theorem}
\begin{proof}
By assumption, there exist closed $A$-invariant subspaces $\HH_1$
and $\HH_2$ of $\HH$ such that $\HH_1 \oplus \HH_2=\HH$,
$A_{\HH_1}=A_1$ is bounded below and $A_{\HH_2}=A_2$ is
quasinilpotent. Also, there exist closed $B$-invariant subspaces
$\KK_1$ and $\KK_2$ of $\KK$ such that $\KK_1 \oplus \KK_2=\KK$,
$B_{\KK_1}=B_1$ is surjective and $B_{\KK_2}=B_2$ is quasinilpotent.
It is clear that
$\beta(A-\lambda)=\beta(A_1-\lambda)+\beta(A_2-\lambda)$ and
$\alpha(B-\lambda)=\alpha(B_1-\lambda)+\alpha(B_2-\lambda)$ for
every $\lambda \in \CC$. Since $A_2$ and $B_2$ are quasinilpotent we
have
\begin{align}
\beta(A-\lambda)=\beta(A_1-\lambda) \label{beta}, \\
\alpha(B-\lambda)=\alpha(B_1-\lambda), \label{alpha}
\end{align}
for every $\lambda \in \CC \setminus \{0\}$. Further, according to
Lemma ˜\ref{lema1} there exists $\epsilon >0$ such that
\begin{equation}\label{const}
\beta(A_1-\lambda) \; \; \text{and} \; \; \alpha(B_1-\lambda) \; \;
\text{are constant for} \; \; |\lambda|<\epsilon.
\end{equation}
Consider $\lambda_0 \in \CC$ such that
$0<|\lambda_0|<\min\{\epsilon, \delta\}$, where $\delta$ is from the
condition (iii). Using \eqref{beta}, \eqref{alpha}, \eqref{const}
and the condition (iii) we obtain
\[\beta(A_1)=\beta(A_1-\lambda_0)=\beta(A-\lambda_0)=\alpha(B-\lambda_0)=\alpha(B_1-\lambda_0)=\alpha(B_1).\]

On the other hand, it is easy to see that $\left( \begin{array}{c} \HH_1 \\
\KK_1 \end{array} \right)$ and $\left( \begin{array}{c} \HH_2 \\
\KK_2 \end{array} \right)$ are closed subspaces of $\left( \begin{array}{c} \HH \\
\KK \end{array} \right)$ and that $\left( \begin{array}{c} \HH_1 \\
\KK_1 \end{array} \right) \oplus \left( \begin{array}{c} \HH_2 \\
\KK_2 \end{array} \right)=\left( \begin{array}{c} \HH \\
\KK \end{array} \right)$. $\HH_1$ and $\KK_1$ are separable Hilbert
spaces in their own right, so from \cite[Theorem 2]{Du} it follows
that there exists an operator $C_1 \in L(\KK_1, \HH_1)$ such that
the operator $\left(
\begin{array}{cc} A_1 & C_1 \\
0 & B_1 \end{array} \right): \left( \begin{array}{c} \HH_1 \\
\KK_1 \end{array} \right) \to \left( \begin{array}{c} \HH_1 \\
\KK_1 \end{array} \right)$ is invertible. Let define an operator $C
\in L(\KK, \HH)$ by
\[C=\left( \begin{array}{cc} C_1 & 0 \\
0 & 0 \end{array} \right): \left( \begin{array}{c} \KK_1 \\
\KK_2 \end{array} \right) \to \left( \begin{array}{c} \HH_1 \\
\HH_2 \end{array} \right).\] An easy computation shows that $\left( \begin{array}{c} \HH_1 \\
\KK_1 \end{array} \right)$ and $\left( \begin{array}{c} \HH_2 \\
\KK_2 \end{array} \right)$ are invariant for $M_C=\left( \begin{array}{cc} A & C \\
0 & B \end{array} \right)$ and also
\begin{align*}
(M_C)_{\HH_1 \oplus \KK_1}=\left( \begin{array}{cc} A_1 & C_1 \\
0 & B_1 \end{array} \right), \\
(M_C)_{\HH_2 \oplus \KK_2}=\left( \begin{array}{cc} A_2 & 0\\
0 & B_2 \end{array} \right).
\end{align*}
Since $A_2$ and $B_2$ are quasinilpotent, from $\sigma((M_C)_{\HH_2
\oplus \KK_2})=\sigma(A_2) \cup \sigma(B_2)=\{0\}$, it follows that
$(M_C)_{\HH_2 \oplus \KK_2}$ is quasinilpotent. Finally,
$(M_C)_{\HH_1 \oplus \KK_1}$ is invertible and thus $M_C$ is
generalized Drazin invertible.
\end{proof}

\noindent Using Theorem ˜\ref{GD} we obtain \cite[Theorem
2.1]{maroko} in a simpler way.

\begin{theorem} [\cite{maroko}] \label{D}
Let $A \in L(\HH)$ and $B \in L(\KK)$ be given operators on
separable Hilbert spaces $\HH$ and $\KK$, respectively, such
that:\par

\medskip

\noindent {\rm (i)} $A$ is left Drazin invertible;
\par \noindent{\rm (ii)} $B$ is right Drazin invertible; \par \noindent {\rm (iii)} There exists a constant $\delta>0$ such
that $\beta(A-\lambda)=\alpha(B-\lambda)$ for every $\lambda \in
\CC$ such that $0<|\lambda|<\delta$. \par

\smallskip

\noindent Then there exists an operator $C \in L(\KK, \HH)$ such
that $M_C$ is Drazin invertible.
\end{theorem}
\begin{proof}
By \cite[Theorem 3.12]{Ber0} it follows that there exist pairs
$(\HH_1, \HH_2) \in Red(A)$ and $(\KK_1, \KK_2) \in Red(B)$ such
that $A_{\HH_1}=A_1$ is bounded below, $B_{\KK_1}=B_1$ is
surjective, $A_{\HH_2}=A_2$ and $B_{\KK_2}=B_2$ are nilpotent. From
the proof of Theorem ˜\ref{GD} we conclude that there exists $C \in
L(\KK, \HH)$ such that
\begin{align*}
M_C=(M_C)_{\HH_1 \oplus \KK_1} \oplus (M_C)_{\HH_2 \oplus \KK_2}, \\
(M_C)_{\HH_1 \oplus \KK_1} \; \text{is invertible},\\
(M_C)_{\HH_2 \oplus \KK_2}=\left( \begin{array}{cc} A_2 & 0\\
0 & B_2 \end{array} \right).
\end{align*}
For sufficiently large $n \in \NN$ we have
\[\left( \begin{array}{cc} A_2 & 0\\
0 & B_2 \end{array} \right)^n=\left( \begin{array}{cc} (A_2)^n & 0\\
0 & (B_2)^n \end{array} \right)=\left( \begin{array}{cc} 0 & 0\\
0 & 0 \end{array} \right). \] It means that $(M_C)_{\HH_2 \oplus
\KK_2}$ is nilpotent, and the proof is complete.
\end{proof}

Under the additional assumptions the converse implications in
Theorem ˜\ref{GD} and Theorem ˜\ref{D} are also true even in the
context of Banach spaces.

\begin{theorem} \label{converse1}
Let $A \in L(\XX)$ and $B \in L(\YY)$ admit a GKD. If there exists
some $C \in L(\YY, \XX)$ such that $M_C$ is generalized Drazin
invertible, then the following holds:

\noindent {\rm (i)} $A=A_1 \oplus A_2$, $A_1$ is bounded below and
$A_2$ is quasinilpotent;
\par \noindent{\rm (ii)} $B=B_1 \oplus B_2$, $B_1$ is surjective and $B_2$ is quasinilpotent; \par \noindent {\rm (iii)} There exists
a constant $\delta>0$ such that $\beta(A-\lambda)=\alpha(B-\lambda)$
for every $\lambda \in \CC$ such that $0<|\lambda|<\delta$. \par
\end{theorem}
\begin{proof}
Let $M_C$ be generalized Drazin invertible for some $C \in L(\YY,
\XX)$. Then there exists $\delta >0$ such that $M_C-\lambda$ is
invertible for $0<|\lambda|<\delta$. According to \cite[Theorem
2]{Han} we have $0 \not \in \acc \; \sigma_l(A)$, $0 \not \in \acc
\; \sigma_r(B)$ and that the statement {\rm (iii)} is satisfied. It
means that also $0 \not \in \acc \; \sigma_{ap}(A)$ and $0 \not \in
\acc \; \sigma_{su}(B)$. By Lemma ˜\ref{lema3} we obtain that the
statements {\rm (i)} and {\rm (ii)} are also satisfied.
\end{proof}
\begin{theorem} \label{converse2}
Let $A \in L(\XX)$ be of Kato type and let $B \in L(\YY)$ admit a
GKD. If there exists some $C \in L(\YY, \XX)$ such that $M_C$ is
Drazin invertible, then the following holds:\par

\noindent {\rm (i)} $A$ is left Drazin invertible; \par
\noindent{\rm (ii)} $B$ is right Drazin invertible; \par \noindent
{\rm (iii)} There exists a constant $\delta>0$ such that
$\beta(A-\lambda)=\alpha(B-\lambda)$ for every $\lambda \in \CC$
such that $0<|\lambda|<\delta$. \par
\end{theorem}
\begin{proof}
Applying the same argument as in Theorem ˜\ref{converse1} we obtain
that the statement {\rm (iii)} holds and $0 \not \in \acc \;
\sigma_{ap}(A) \cup \acc \; \sigma_{su}(B)$. Now Lemma ˜\ref{lema3}
implies that there exist $(X_1, X_2) \in Red(A)$ and $(Y_1, Y_2) \in
Red(B)$ such that $A_1$ is bounded below, $A_2$ is nilpotent, $B_1$
is surjective and $B_2$ is quasinilpotent. Since $\dsc(B)<\infty$
\cite[Lemma 2.6]{Zhong}, then $\dsc(B_2)<\infty$, so $B_2$ is
quasinilpotent operator with finite descent. We conclude that $B_2$
is nilpotent by \cite[Corollary 10.6, p. 332]{TL}. Let $n \geq d$
where $d \in \NN$ is such that $(A_2)^d=0$ and $(A_2)^{d-1} \neq 0$.
We have
\begin{align*}
N(A^n)=N((A_1)^n) \oplus N((A_2)^n)=\XX_2, \\
N(A^{d-1})=N((A_1)^{d-1}) \oplus N((A_2)^{d-1})=N((A_2)^{d-1})
\subsetneq \XX_2.
\end{align*}
It follows that $\asc(A)=d<\infty$. From $R(A^n)=R((A_1)^n) \oplus
R((A_2)^n)=R((A_1)^n)$ we conclude that $R(A^n)$ is closed and
therefore $A$ is left Drazin invertible. In a similar way we prove
that $B$ is right Drazin invertible.
\end{proof}

\noindent An operator $T \in L(\XX)$ is {\em semi-Fredholm} if
$R(T)$ is closed and if $\alpha(T)$ or $\beta(T)$ is finite. The
class of semi-Fredholm operators belongs to the class of Kato type
operators \cite[Theorem 16.21]{Mu}. According to this observation it
seems that Theorem ˜\ref{converse2} is an extension of
\cite[Corollary 2.3]{maroko}.

If $\delta > 0$, set $\DD(0, \delta)=\{\lambda \in \CC:
|\lambda|<\delta\}$. The following theorem is our second main
result.

\begin{theorem} \label{KolihaDrazin}
Let $A \in L(\XX)$ and $B \in L(\YY)$ be given operators such that
$0 \not \in \acc \, \sigma_d(A)$ or $0 \not \in \acc \,
\sigma_p(B)$.

\noindent {\rm (i)} If there exists some $C \in L(\YY, \XX)$ such
that $M_C$ is generalized Drazin invertible, then $A$ and $B$ are
both generalized Drazin invertible.\par

\noindent {\rm (ii)} If there exists some $C \in L(\YY, \XX)$ such
that $M_C$ is Drazin invertible, then $A$ and $B$ are both Drazin
invertible.\par

\noindent {\rm (iii)} If there exists some $C \in L(\YY, \XX)$ such
that $M_C$ is invertible, then $A$ and $B$ are both invertible.
\end{theorem}
\begin{proof}
{\rm (i)}. Suppose that $0 \not \in \acc \, \sigma_d(A)$ and that
$M_C$ is generalized Drazin invertible for some $C \in L(\YY, \XX)$.
Since $0 \not \in \acc \, \sigma(M_C)$ then there exists $\delta >
0$ such that
\begin{align}
M_C-\lambda \; \text{is invertible}, \label{eq1}\\
\overline{R(A-\lambda)}=\XX, \label{eq2}
\end{align}
for every $\lambda \in \DD(0, \delta)\setminus \{0\}$. From
\cite[Theorem 2]{Han} it follows that $A-\lambda$ is left invertible
for every $\lambda \in \DD(0, \delta)\setminus \{0\}$. Since
$A-\lambda$ has closed range for every $\lambda \in \DD(0,
\delta)\setminus \{0\}$, from ˜\eqref{eq2} we conclude
$R(A-\lambda)=\overline{R(A-\lambda)}=\XX$ for $\lambda \in \DD(0,
\delta)\setminus \{0\}$, i.e. $A-\lambda$ is surjective for $\lambda
\in \DD(0, \delta)\setminus \{0\}$. Moreover, $A-\lambda$ is
injective for $\lambda \in \DD(0, \delta)\setminus \{0\}$, hence
$A-\lambda$ is invertible for $\lambda \in \DD(0, \delta)\setminus
\{0\}$. It means that $0 \not \in \acc \, \sigma(A)$, so $A$ is
generalized Drazin invertible. Now, \eqref{eq1} and Lemma
˜\ref{invertibility} imply that $B$ is generalized Drazin
invertible.

Assume that $0 \not \in \acc \, \sigma_p(B)$ and that $M_C$ is
generalized Drazin invertible for some $C \in L(\YY, \XX)$. There
exists $\delta>0$ such that ˜\eqref{eq1} holds and
\begin{equation}\label{eq3}
B-\lambda \; \text{is injective for} \; \lambda \in \DD(0,
\delta)\setminus \{0\}.
\end{equation}
$B$ is generalized Drazin invertible by ˜\eqref{eq1}, ˜\eqref{eq3}
and \cite[Theorem 2]{Han}. We apply Lemma ˜\ref{invertibility} again
and obtain that $A$ is generalized Drazin invertible.\par

\smallskip

\noindent{\rm (ii)}. If there exists $C \in L(\YY,\XX)$ such that
$M_C$ is Drazin invertible, then $M_C$ is also generalized Drazin
invertible. According to the part (i) it follows that $B$ is
generalized Drazin invertible, i.e. $0 \not \in \acc \, \sigma(B)$.
Further, $\dsc(B)<\infty$ by \cite[Lemma 2.6]{Zhong}, so from
\cite[Corollary 1.6]{dsc} it follows that $B$ is Drazin invertible.
We apply \cite[Lemma 2.7]{Zhong} and obtain that $A$ is also Drazin
invertible.
\par

\smallskip

\noindent{\rm (iii}). $A$ is left invertible and $B$ is right
invertible by \cite[Theorem 2]{Han}. On the other hand, the part (i)
now leads to $0 \not \in \acc \, \sigma(A)$ and $0 \not \in \acc \,
\sigma(B)$. The result follows from \eqref{leftsemi} and
\eqref{rightsemi}.
\end{proof}

\noindent The part {\rm (ii)} of Theorem ˜\ref{KolihaDrazin} is also
an extension of \cite[Corollary 2.3]{maroko}. The following result
is an immediate consequence of Theorem ˜\ref{KolihaDrazin}. The
second inclusion of Corollary ˜\ref{inkluzije} is proved in
\cite[Corollary 4]{Han}, \cite[Theorem 5.1]{Dragan} and \cite[Lemma
2.3]{Dragan1}.

\begin{corollary} \label{inkluzije}
Let $A \in L(\XX)$ and $B \in L(\YY)$. If $\sigma_{\ast} \in
\{\sigma, \sigma_{D}, \sigma_{gD}\}$, then we have
\[(\sigma_{\ast}(A) \cup \sigma_{\ast}(B))\setminus (\acc \,
\sigma_d(A) \cap \acc \, \sigma_p(B)) \subset \sigma_{\ast}(M_C)
\subset \sigma_{\ast}(A) \cup \sigma_{\ast}(B) \] for every $C \in
L(\YY, \XX)$.
\end{corollary}
\begin{remark} \label{remark}
{\em Let $T \in L(\XX)$. Inclusions $\acc \, \sigma_p(T) \subset
\acc \, \sigma_l(T) \subset \acc \, \sigma(T) \subset \sigma(T)$ and
$\acc \, \sigma_d(T) \subset \acc \, \sigma_r(T) \subset \acc \,
\sigma(T) \subset \sigma(T)$ are clear. According to \cite[Theorem
12(iii)]{Bo} and from the fact that the Drazin spectrum is compact
we have $\acc \, \sigma_p(T) \subset \acc \, \sigma(T)=\acc \,
\sigma_D(T) \subset \sigma_D(T)$ and also $\acc \, \sigma_d(T)
\subset \acc \, \sigma(T)=\acc \, \sigma_D(T) \subset \sigma_D(T)$.
From this consideration we obtain the following statements for every
$A \in L(\XX)$ and $B \in L(\YY)$:

\noindent {\rm (i)} $\acc \, \sigma_d(A) \cap \acc \, \sigma_p(B)
\subset \sigma(A) \cap \sigma(B)$;\par \noindent {\rm (ii)} $\acc \,
\sigma_d(A) \cap \acc \, \sigma_p(B) \subset \sigma_D(A) \cap
\sigma_D(B)$;\par \noindent {\rm (iii)} $\acc \, \sigma_d(A) \cap
\acc \, \sigma_p(B) \subset \acc \, \sigma_r(A) \cap \acc \,
\sigma_l(B)\subset \sigma_r(A) \cap \sigma_l(B)$. }
\end{remark}
\noindent Remark ˜\ref{remark} shows that Corollary ˜\ref{inkluzije}
is an improvement of \cite[Corollary 4]{Han}, \cite[Corollary
3.17]{pseudoBW} and of the equation (1) in \cite{Zhong}. Another
extension of the equation (1) from \cite{Zhong} may be found in
\cite[Proposition 2.2]{maroko}.

We also generalize \cite[Corollary 2.6]{maroko}.

\begin{corollary} \label{cormar}
Let $A \in L(\XX)$, $B \in L(\YY)$ and let $\sigma_{\ast} \in
\{\sigma, \sigma_{D}, \sigma_{gD}\}$. If $\acc \, \sigma_d(A) \cap
\acc \, \sigma_p(B)=\emptyset$ then
\begin{equation} \label{formula}
\sigma_{\ast}(M_C)=\sigma_{\ast}(A) \cup \sigma_{\ast}(B) \; \;
\text{for every} \; \; C \in L(\YY, \XX).
\end{equation}
In particular, if $\sigma_r(A) \cap \sigma_l(B)=\emptyset$ then
\eqref{formula} is satisfied.
\end{corollary}
\begin{proof}
Apply Corollary ˜\ref{inkluzije} and {\rm (iii)} of Remark
˜\ref{remark}.
\end{proof}

\begin{example}
{\em Let $\XX=\YY=\ell^2(\NN)$ and let $U$ be forward unilateral
shift operator on $\ell^2(\NN)$. It is known that
\[\sigma(U)=\sigma_D(U)=\sigma_{gD}(U)=\{\lambda \in \CC:|\lambda| \leq 1\}.\]
Since $\sigma_p(U)=\emptyset$, then from Corollary ˜\ref{cormar} we
conclude
\[\sigma_{\ast}\left( \left(\begin{array}{cc} A & C \\
0 & U\end{array} \right) \right)=\sigma_{\ast}(A) \cup
\sigma_{\ast}(U),\] where $\sigma_{\ast} \in \{\sigma, \sigma_D,
\sigma_{gD} \}$ and $A, C \in L(\ell^2(\NN))$ are arbitrary
operators. }
\end{example}

\noindent We finish with a slight extension of \cite[Corollary
7]{Han}, \cite[Theorem 2.9]{Zhong} and \cite[Theorem
3.18]{pseudoBW}.

\begin{theorem} \label{holest}
Let $A \in L(\XX)$ and $B \in L(\YY)$. If $\sigma_{\ast} \in
\{\sigma, \sigma_{D}, \sigma_{gD}\}$, then for every $C \in L(\YY,
\XX)$ we have
\begin{equation} \label{holes}
\sigma_{\ast}(A) \cup \sigma_{\ast}(B)=\sigma_{\ast}(M_C) \cup W,
\end{equation}
where $W$ is the union of certain holes in $\sigma_{\ast}(M_C)$,
which happen to be subsets of $\acc \, \sigma_d(A) \cap \acc \,
\sigma_p(B)$.
\end{theorem}
\begin{proof}
Expression \eqref{holes} is the main result of \cite{Han},
\cite{Zhong} and \cite{Zhang}. Corollary ˜\ref{inkluzije} shows that
the filling some holes in $\sigma_{\ast}(M_C)$ should occur in $\acc
\, \sigma_d(A) \cap \acc \, \sigma_p(B)$.
\end{proof}

\medskip

\noindent {\bf Acknowledgements.} The author is supported by the
Ministry of Education, Science and Technological Development,
Republic of Serbia, grant no. 174007.

\bigskip
\noindent
\author{Milo\v s D. Cvetkovi\'c}

\noindent University of Ni\v s\\
Faculty of Sciences and Mathematics\\
P.O. Box 224, 18000 Ni\v s, Serbia

\noindent {\it E-mail}: {\tt miloscvetkovic83@gmail.com}
\end{document}